\documentclass[11pt,a4paper]{article}
\usepackage[T1]{fontenc}
\usepackage{lmodern}
\usepackage{microtype}
\usepackage[margin=1in]{geometry}
\usepackage{amsmath,amssymb,amsthm,mathtools}
\usepackage{booktabs,array}
\usepackage{hyperref}
\usepackage{cite}
\newtheorem{theorem}{Theorem}[section]

\newtheorem{lemma}[theorem]{Lemma}

\newcommand{\ii}{\mathrm i}
\newcommand{\Cl}{\operatorname{Cl}}

\title{A Four-Genus Kronecker-Limit Evaluation of the\\ Alternating Rogers--Ramanujan Continued Fraction}
\author{Sumit Kumar Jha}
\date{}
\begin{document}
\maketitle

\begin{abstract}
We evaluate the six odd class-number-four cases left unevaluated in Ramanathan's treatment of the Rogers--Ramanujan continued fraction.  Let
\[
 S(q)=-R(-q),\qquad
 R(q)=\cfrac{q^{1/5}}{1+\cfrac{q}{1+\cfrac{q^2}{1+\cfrac{q^3}{1+\cdots}}}}.
\]
For $n=39,87,111,119,159,287$ we determine the value of
$S(e^{-\pi/\sqrt{5n}})$ by applying the genus-character form of the Kronecker limit formula to the four ideal classes of $\mathbb Q(\sqrt{-5n})$.  The resulting expressions are given in terms of fundamental units of real quadratic fields.  In particular, if $X_n=S(e^{-\pi/\sqrt{5n}})$, then
\[
 X_n^{-5}+11-X_n^5=\frac{5\sqrt5}{U_n},
\]
where the six quantities $U_n$ are displayed explicitly below.  We also give the corresponding radical expressions and quartic algebraic certificates, together with numerical checks of the six evaluations.
\end{abstract}

\section{Introduction}
The Rogers--Ramanujan continued fraction
\[
R(q)=\cfrac{q^{1/5}}{1+\cfrac{q}{1+\cfrac{q^2}{1+\cfrac{q^3}{1+\cdots}}}},
\qquad |q|<1,
\]
has a long history of explicit evaluations at quadratic arguments.  We shall use the alternating form
\[
S(q)=-R(-q).
\]

Ramanathan reduced the evaluation of special values of this function to eta quotients and the Kronecker limit formula.  In particular, he listed
\[
39,\ 87,\ 111,\ 119,\ 159,\ 287
\]
as the odd cases in which the class number is four and each genus contains one class, but did not carry out their explicit evaluations; the following even case was used instead to illustrate the method \cite[p.~72]{Ramanathan1984}.  We carry out the corresponding four-genus calculation for these six cases.

The calculation combines the product identities for the Rogers--Ramanujan continued fraction, Ramanathan's eta transformation formulas, and the genus-character form of the Kronecker limit formula.  In the present four-class situation the latter reduces to a two-term Fourier calculation on the genus group.

The resulting unit expressions involve the real quadratic discriminants
\[
5,13,17,37,53,65,85,145,185,205,265.
\]
The class numbers of these fields are not all one; in particular
$h(65)=h(85)=h(185)=h(205)=h(265)=2$ and $h(145)=4$.  These values are essential for the exponents occurring in the final formulas \cite{Cohen1993}.

Explicit CM evaluations of the Rogers--Ramanujan continued fraction and related class invariants have also been obtained by Berndt, Chan and Zhang by eta-function and modular-equation methods \cite{BCZ1996,BCZ1997}.  We use Ramanathan's Kronecker-limit framework to complete these six four-genus evaluations.

\section{The eta quotient at the CM point}
The classical product identities are
\begin{align}
\frac1{R(q)}-1-R(q)
 &=q^{-1/5}\frac{(q^{1/5};q^{1/5})_\infty}{(q^5;q^5)_\infty},\label{eq:R1}\\
\frac1{R(q)^5}-11-R(q)^5
 &=q^{-1}\frac{(q;q)_\infty^6}{(q^5;q^5)_\infty^6}.\label{eq:R5}
\end{align}
These identities are classical; see Ramanathan \cite{Ramanathan1984} and Berndt--Chan--Zhang \cite{BCZ1996}.  Replacing $q$ by $-q$, with the real fifth-root convention, gives for $0<q<1$
\begin{align}
\frac1{S(q)}+1-S(q)
 &=q^{-1/5}\frac{(q^{1/5};q^{1/5})_\infty}{(q^5;q^5)_\infty},\label{eq:S1}\\
\frac1{S(q)^5}+11-S(q)^5
 &=q^{-1}\frac{(q;q)_\infty^6}{(q^5;q^5)_\infty^6}.\label{eq:S5}
\end{align}

Ramanathan introduces a shifted function by
\[
-S(\tau)=R\!\left(\frac{\tau+5}{2}\right)
\]
\cite[p.~68]{Ramanathan1984}.  Since the nome corresponding to $(\tau+5)/2$ is
$-e^{\pi i\tau}$, his $S(\tau)$ agrees with the present notation
$S(e^{\pi i\tau})=-R(-e^{\pi i\tau})$.

We now specialize to the CM points used by Ramanathan.  Put
\[
\mathcal S(\tau)=S(e^{\pi i\tau}).
\]
Let $n>0$ be odd, let $D=-5n$ be a fundamental discriminant, and put
\[
\Omega=\frac{1+\sqrt D}{2}.
\]
Then Ramanathan's equation (14) \cite[p.~70]{Ramanathan1984} gives, in the present notation,
\begin{equation}
\mathcal S\!\left(\frac{\ii}{\sqrt{5n}}\right)^{-5}+11-
\mathcal S\!\left(\frac{\ii}{\sqrt{5n}}\right)^5
=5^3\left|\frac{\eta(\Omega)}
{\eta((1+\sqrt D/5)/2)}\right|^6.
\label{eq:Ramanathan14}
\end{equation}
We shall combine this identity with the genus-character form of the Kronecker limit formula.

For comparison with the theta-product formulation, put
\[
f(\tau)=q^{-1/24}\prod_{m\ge1}(1+q^{2m-1}),
\qquad q=e^{\pi i\tau}.
\]
Jacobi's product gives
\[
f(\tau)^2=\frac{\vartheta_3(0\mid\tau)}{\eta(\tau)}.
\]
Thus, on defining
\[
T(\tau)=\eta(\tau)\vartheta_3(0\mid\tau),
\]
Ramanathan's equation (6) \cite[p.~68]{Ramanathan1984}, together with Jacobi's product, can equivalently be written
\begin{equation}
\frac1{\mathcal S(\tau)^5}+11-\mathcal S(\tau)^5
=\left(\frac{T(\tau)}{T(5\tau)}\right)^3.
\label{eq:Tidentity}
\end{equation}
The standard transformations of $\eta$ and $\vartheta_3$ then imply
\[
T(-1/\tau)=(-i\tau)T(\tau).
\]

\section{The four-genus Kronecker-limit calculation}
Let $K=\mathbb Q(\sqrt D)$, with $D<0$ fundamental, and let $A$ be an ideal class represented by
$[a,b+\Omega]$, where $\Omega=(1+\sqrt D)/2$.  Put
\[
z=\frac{b+\Omega}{a},
\qquad
F(A)=\frac{|\eta(z)|^2}{\sqrt a}.
\]
The genus-character form of the Kronecker limit formula, in the form used by Berndt--Chan--Zhang (see their equation (2.18)) \cite{BCZ1997}, states that
\begin{equation}
\frac{\sqrt{|D|}}{2\pi}L_K(1,\chi)
=-\sum_{A\in\Cl(K)}\chi(A)\log F(A),
\label{eq:siegel}
\end{equation}
where $\chi$ is a nonprincipal genus character and $L_K(s,\chi)$ denotes the corresponding class-group $L$-series.

Suppose now that $h(D)=4$ and that each genus contains one class.  Let $A$ denote the principal class and let
\[
C=[5,2+\Omega].
\]
Since
\[
\frac{2+\Omega}{5}=\frac{1+\sqrt D/5}{2},
\]
we have
\[
F(A)=|\eta(\Omega)|^2,
\qquad
F(C)=\frac{\left|\eta((1+\sqrt D/5)/2)\right|^2}{\sqrt5}.
\]
Moreover the corresponding reduced form is
\[
\left[5,5,\frac{n+5}{4}\right].
\]

Set
\[
\Lambda_\chi=\frac{\sqrt{|D|}}{2\pi}L_K(1,\chi).
\]
We shall need the following consequence of \eqref{eq:siegel}.

\begin{lemma}
If $A$ is the principal class and $C$ is the norm-$5$ class, then
\begin{equation}
\log F(A)-\log F(C)
=-\frac12\sum_{\substack{\chi\ne1\\ \chi(C)=-1}}\Lambda_\chi.
\label{eq:fourier}
\end{equation}
\end{lemma}

\begin{proof}
The four genera form the group $(\mathbb Z/2\mathbb Z)^2$.  There are exactly two nonprincipal characters which are $-1$ on $C$.  Summing \eqref{eq:siegel} over these two characters, the contributions of the two classes other than $A$ and $C$ cancel, while the contributions of $A$ and $C$ are respectively $+2\log F(A)$ and $-2\log F(C)$.  This gives \eqref{eq:fourier}.
\end{proof}

Let
\[
X_n=\mathcal S\!\left(\frac{\ii}{\sqrt{5n}}\right)=S(e^{-\pi/\sqrt{5n}}),
\qquad
C_n=X_n^{-5}+11-X_n^5.
\]
By Ramanathan's equation (14) \cite[p.~70]{Ramanathan1984} and the preceding expressions for $F(A)$ and $F(C)$, we obtain
\begin{equation}
C_n=5\sqrt5\left(\frac{F(A)}{F(C)}\right)^3.
\label{eq:CF}
\end{equation}
Consequently, by the lemma,
\begin{equation}
C_n=\frac{5\sqrt5}{U_n},
\qquad
\log U_n=\frac32\sum_{\substack{\chi\ne1\\ \chi(C)=-1}}\Lambda_\chi.
\label{eq:Umaster}
\end{equation}
It remains to determine the two genus characters occurring in each of the six cases.

\section{The six discriminants}
Each of the six integers $D=-5n$ is a fundamental discriminant.  Each of the six discriminants $D=-5n$ is fundamental.  The reduced primitive positive definite forms are found from
$|b|\le a\le c$ and $b^2-4ac=D$; since $a\le\sqrt{|D|/3}$, the enumeration is finite.  The complete lists are
\[
\begin{array}{c|c|l}
 n&D&\text{reduced forms}\\ \hline
39&-195&[1,1,49],\ [3,3,17],\ [5,5,11],\ [7,1,7]\\
87&-435&[1,1,109],\ [3,3,37],\ [5,5,23],\ [11,7,11]\\
111&-555&[1,1,139],\ [3,3,47],\ [5,5,29],\ [13,11,13]\\
119&-595&[1,1,149],\ [5,5,31],\ [7,7,23],\ [13,9,13]\\
159&-795&[1,1,199],\ [3,3,67],\ [5,5,41],\ [15,15,17]\\
287&-1435&[1,1,359],\ [5,5,73],\ [7,7,53],\ [19,3,19].
\end{array}
\]
Thus all six class numbers are equal to $4$.  Since each discriminant has three prime-discriminant factors, genus theory gives four genera, so each genus consists of one class \cite{BCZ1997}.

The relevant prime-discriminant decompositions are
\[
\begin{array}{c|c}
39&(-3)\cdot5\cdot13\\
87&(-3)\cdot5\cdot29\\
111&(-3)\cdot5\cdot37\\
119&(-7)\cdot5\cdot17\\
159&(-3)\cdot5\cdot53\\
287&(-7)\cdot5\cdot41.
\end{array}
\]
For a reduced form $[a,b,c]$ and a prime-discriminant factor $d$ with $(c,d)=1$, the genus character is $(d/c)$ \cite{BCZ1997}.  Applying this to the four classes gives the character vectors below; the order of the three signs follows the order of the displayed prime-discriminant factors.
\[
\begin{array}{c|c|c}
 n&\text{character vectors}&\text{vector of }C\\ \hline
39&(+,+,+),(-,-,+),(-,+,-),(+,-,-)&(-,+,-)\\
87&(+,+,+),(+,-,-),(-,-,+),(-,+,-)&(-,-,+)\\
111&(+,+,+),(-,-,+),(-,+,-),(+,-,-)&(-,+,-)\\
119&(+,+,+),(-,+,-),(+,-,-),(-,-,+)&(-,+,-)\\
159&(+,+,+),(+,-,-),(-,+,-),(-,-,+)&(-,+,-)\\
287&(+,+,+),(-,-,+),(+,-,-),(-,+,-)&(-,-,+).
\end{array}
\]
For example, when $n=39$, $C=[5,5,11]$ and
\[
\left(\frac{-3}{11}\right)=-1,
\qquad
\left(\frac5{11}\right)=1,
\qquad
\left(\frac{13}{11}\right)=-1.
\]
Thus the two factorizations selected by $\chi(C)=-1$ are
\[
\begin{array}{c|c}
39&(-3,65),\ (13,-15)\\
87&(-3,145),\ (5,-87)\\
111&(-3,185),\ (37,-15)\\
119&(-7,85),\ (17,-35)\\
159&(-3,265),\ (53,-15)\\
287&(-7,205),\ (5,-287).
\end{array}
\]

\section{Quadratic L-values}
For a genus character associated with a factorization $D=d_1d_2$, the genus $L$-series factors as \cite{BCZ1997}
\begin{equation}
L_K(s,\chi)=L(s,\chi_{d_1})L(s,\chi_{d_2}).
\label{eq:Lfactor}
\end{equation}
The standard class-number formulas (see, for example, Cohen \cite{Cohen1993}) are
\begin{align}
L(1,\chi_d)&=\frac{2\pi h(d)}{w_d\sqrt{|d|}},&&d<0,\label{eq:Lneg}\\
L(1,\chi_d)&=\frac{2h(d)\log\epsilon_d}{\sqrt d},&&d>0,
\label{eq:Lpos}
\end{align}
where $\epsilon_d>1$ is a fundamental unit of the real quadratic field of discriminant $d$.

The negative class numbers required here are
\[
\begin{array}{c|rrrrrrr}
d&-3&-7&-15&-35&-87&-119&-287\\ \hline
h(d)&1&1&2&2&6&10&14.
\end{array}
\]
Here $w_{-3}=6$ and $w_d=2$ for the remaining discriminants; the class numbers may also be found in Cohen, \S B.1 \cite{Cohen1993}.

For the positive discriminants one has
\[
\begin{array}{c|rrrrrrrrrrr}
d&5&13&17&37&53&65&85&145&185&205&265\\ \hline
h(d)&1&1&1&1&1&2&2&4&2&2&2.
\end{array}
\]
These class-number values are listed in Cohen \cite{Cohen1993}.  The corresponding fundamental units are
\[
\begin{array}{c|c}
d&\epsilon_d\\ \hline
5&(1+\sqrt5)/2\\
13&(3+\sqrt{13})/2\\
17&4+\sqrt{17}\\
37&6+\sqrt{37}\\
53&(7+\sqrt{53})/2\\
65&8+\sqrt{65}\\
85&(9+\sqrt{85})/2\\
145&12+\sqrt{145}\\
185&68+5\sqrt{185}\\
205&(43+3\sqrt{205})/2\\
265&6072+373\sqrt{265}.
\end{array}
\]

Substitution of \eqref{eq:Lfactor}--\eqref{eq:Lpos} into $\Lambda_\chi$ yields
\[
\begin{array}{c|c|c}
 n&\chi&\Lambda_\chi\\ \hline
39&(-3,65)&\frac23\log\epsilon_{65}\\
 &(13,-15)&2\log\epsilon_{13}\\[1mm]
87&(-3,145)&\frac43\log\epsilon_{145}\\
 &(5,-87)&6\log\epsilon_5\\[1mm]
111&(-3,185)&\frac23\log\epsilon_{185}\\
 &(37,-15)&2\log\epsilon_{37}\\[1mm]
119&(-7,85)&2\log\epsilon_{85}\\
 &(17,-35)&2\log\epsilon_{17}\\[1mm]
159&(-3,265)&\frac23\log\epsilon_{265}\\
 &(53,-15)&2\log\epsilon_{53}\\[1mm]
287&(-7,205)&2\log\epsilon_{205}\\
 &(5,-287)&14\log\epsilon_5.
\end{array}
\]
For example,
\[
\Lambda_{(-3,65)}
=\frac{\sqrt{195}}{2\pi}L(1,\chi_{-3})L(1,\chi_{65})
=\frac23\log\epsilon_{65}.
\]
The value $h(65)=2$ is responsible for the coefficient $2/3$.

It follows from \eqref{eq:Umaster} that
\[
\begin{array}{c|c}
 n&U_n\\ \hline
39&\epsilon_{65}\epsilon_{13}^{3}\\
87&\epsilon_{145}^{2}\epsilon_5^{9}\\
111&\epsilon_{185}\epsilon_{37}^{3}\\
119&\epsilon_{85}^{3}\epsilon_{17}^{3}\\
159&\epsilon_{265}\epsilon_{53}^{3}\\
287&\epsilon_{205}^{3}\epsilon_5^{21}.
\end{array}
\]

\section{The six evaluations}
\begin{theorem}\label{thm:main}
Let
\[
X_n=\mathcal S\!\left(\frac{\ii}{\sqrt{5n}}\right)=S(e^{-\pi/\sqrt{5n}}),
\qquad n\in\{39,87,111,119,159,287\}.
\]
Then
\begin{equation}
X_n^{-5}+11-X_n^5=\frac{5\sqrt5}{U_n},
\label{eq:main}
\end{equation}
where
\[
\begin{aligned}
U_{39}&=\epsilon_{65}\epsilon_{13}^{3},&
U_{87}&=\epsilon_{145}^{2}\epsilon_5^{9},\\
U_{111}&=\epsilon_{185}\epsilon_{37}^{3},&
U_{119}&=\epsilon_{85}^{3}\epsilon_{17}^{3},\\
U_{159}&=\epsilon_{265}\epsilon_{53}^{3},&
U_{287}&=\epsilon_{205}^{3}\epsilon_5^{21}.
\end{aligned}
\]
Equivalently,
\begin{equation}
X_n=\left[
\frac{11-5\sqrt5/U_n+\sqrt{(11-5\sqrt5/U_n)^2+4}}{2}
\right]^{1/5}.
\label{eq:radical}
\end{equation}
\end{theorem}

\begin{proof}
Equation \eqref{eq:main} follows from \eqref{eq:Umaster} and the preceding table of $\Lambda_\chi$.  If $Y=X_n^5$, then
\[
Y^{-1}+11-Y=\frac{5\sqrt5}{U_n},
\]
and hence
\[
Y^2-\left(11-\frac{5\sqrt5}{U_n}\right)Y-1=0.
\]
Since $Y>0$, the positive root gives \eqref{eq:radical}.
\end{proof}

\section{Algebraic equations for the CM eta-quotient values}
Put
\[
C_n=X_n^{-5}+11-X_n^5=\frac{5\sqrt5}{U_n}.
\]
The element $C_n$ belongs to the biquadratic field
\[
K_n=\begin{cases}
\mathbb Q(\sqrt5,\sqrt{13}),&n=39,\\
\mathbb Q(\sqrt5,\sqrt{29}),&n=87,\\
\mathbb Q(\sqrt5,\sqrt{37}),&n=111,\\
\mathbb Q(\sqrt7,\sqrt{17}),&n=119,\\
\mathbb Q(\sqrt5,\sqrt{53}),&n=159,\\
\mathbb Q(\sqrt5,\sqrt{41}),&n=287.
\end{cases}
\]
Taking the norm of $T-C_n$ from $K_n$ to $\mathbb Q$ gives the following irreducible quartic polynomials:
\begin{align*}
P_{39}(T)&=T^4-6500T^3+194750T^2-812500T+15625,\\
P_{87}(T)&=T^4-491300T^3+42482750T^2-61412500T+15625,\\
P_{111}(T)&=T^4-2682500T^3+391274750T^2-335312500T+15625,\\
P_{119}(T)&=T^4-4530500T^3+107354750T^2-566312500T+15625,\\
P_{159}(T)&=T^4-49422500T^3+18451154750T^2-6177812500T+15625,\\
P_{287}(T)&=T^4-21721789700T^3+862492682750T^2\\
&\qquad -2715223712500T+15625.
\end{align*}
Thus $P_n(C_n)=0$ in each case.  These quartics are algebraic certificates for the exact values of the eta-quotient expression $C_n$; the continued-fraction values $X_n$ are then obtained from the quadratic relation in Theorem~\ref{thm:main}.

\section{Numerical checks}
For a numerical check, the continued fraction defining $S(q)$ was evaluated directly at the six CM points by backward recursion with 3000 levels using 80-digit arithmetic.  The values are shown in Table~\ref{tab:numerical}.

\begin{table}[ht]
\centering
\renewcommand{\arraystretch}{1.15}
\caption{Numerical verification of the six evaluations.}
\label{tab:numerical}
\begin{tabular}{c|c|c|c}
$n$ & $X_n$ & $X_n^{-5}+11-X_n^5$ & $5\sqrt5/U_n$\\ \hline
39  & 1.61747440104808408048 & $1.932018134346904\times10^{-2}$ & $1.932018134346904\times10^{-2}$\\
87  & 1.61802662318246962602 & $2.544718257853804\times10^{-4}$ & $2.544718257853804\times10^{-4}$\\
111 & 1.61803263991982851581 & $4.660085653699190\times10^{-5}$ & $4.660085653699190\times10^{-5}$\\
119 & 1.61803319015009087724 & $2.759091795545396\times10^{-5}$ & $2.759091795545396\times10^{-5}$\\
159 & 1.61803391554312383492 & $2.529231509091291\times10^{-6}$ & $2.529231509091291\times10^{-6}$\\
287 & 1.61803398858333242190 & $5.754590296419898\times10^{-9}$ & $5.754590296419898\times10^{-9}$
\end{tabular}
\end{table}

At 80-digit precision, the differences between the last two columns for the six values of $n$, in the order displayed, were approximately
\[
-1.2\times10^{-80},\quad -5.0\times10^{-80},\quad
2.4\times10^{-80},\quad -1.9\times10^{-80},\quad
-1.1\times10^{-79},\quad 7.9\times10^{-80}.
\]
The agreement provides an independent numerical check of the six evaluations.

\section{Concluding remarks}
The six discriminants considered here are precisely the odd class-number-four cases listed by Ramanathan but left unevaluated.  The four-class structure makes the genus-character part of the Kronecker-limit calculation particularly simple: only two nonprincipal genus characters contribute to the quotient of the principal and norm-$5$ eta invariants.  Their $L$-values reduce to products of quadratic Dirichlet $L$-values, and the real quadratic class numbers determine the exponents of the fundamental units.

The resulting formulas complement the explicit evaluations obtained by other methods in the subsequent literature \cite{BCZ1996,BCZ1997}.  The calculation gives the genus characters, class-number factors, and unit exponents explicitly in each of the six cases.

\appendix
\section{Fundamental units}
For completeness, the fundamental units used above can be verified by the generalized Pell equations.  If $d\equiv1\pmod4$ is a positive fundamental discriminant, an algebraic integer
\[
\epsilon=\frac{x+y\sqrt d}{2}
\]
is a unit precisely when
\[
x^2-dy^2=\pm4.
\]
The continued fractions of $\sqrt d$ give the following minimal solutions:
\[
\begin{array}{c|c|c|c}
d&(x,y)&x^2-dy^2&\epsilon_d\\ \hline
5&(1,1)&-4&(1+\sqrt5)/2\\
13&(3,1)&-4&(3+\sqrt{13})/2\\
17&(8,2)&-4&4+\sqrt{17}\\
37&(12,2)&-4&6+\sqrt{37}\\
53&(7,1)&-4&(7+\sqrt{53})/2\\
65&(16,2)&-4&8+\sqrt{65}\\
85&(9,1)&-4&(9+\sqrt{85})/2\\
145&(24,2)&-4&12+\sqrt{145}\\
185&(136,10)&-4&68+5\sqrt{185}\\
205&(43,3)&4&(43+3\sqrt{205})/2\\
265&(12144,746)&-4&6072+373\sqrt{265}.
\end{array}
\]
The corresponding continued-fraction periods are
\[
\begin{array}{c|l}
d&\text{period of }\sqrt d\\ \hline
5&(4)\\
13&(1,1,1,1,6)\\
17&(8)\\
37&(12)\\
53&(3,1,1,3,14)\\
65&(16)\\
85&(4,1,1,4,18)\\
145&(24)\\
185&(1,1,1,1,26)\\
205&(3,6,1,4,1,6,3,28)\\
265&(3,1,1,2,2,1,1,3,32).
\end{array}
\]
These data give the stated fundamental units by the usual continued-fraction criterion for Pell equations.

\end{document}